\documentclass[reqno,11pt]{amsart}
\usepackage{amssymb, amsmath}
\usepackage{amsthm, amsfonts,mathrsfs}
\usepackage[T1]{fontenc}
\usepackage{times}
\usepackage{graphicx}
\usepackage{subfigure}
\usepackage{epsfig}
\usepackage{color}
\usepackage[all]{xy}
\newtheorem{theorem}{Theorem}[section]

\newtheorem{remark}{Remark}[section]

\numberwithin{equation}{section}

\newcommand{\mr}{\mathbb{R}}

\begin{document}
\title[Exact solutions for nonlinear trapped lee waves]
{Exact solutions for nonlinear trapped lee waves in the $\beta$-plane approximation}%
\author[Lili Fan$^{\dag}$]{Lili Fan$^{\dag}$}%
\address[Lili Fan]{College of Mathematics and Information Science,
Henan Normal University, Xinxiang 453007, China}
\email{fanlily89@126.com\, (Corresponding author)}
\author[Ruonan Liu]{Ruonan Liu}
\address[Ruonan Liu]{College of Mathematics and Information Science,
Henan Normal University, Xinxiang 453007, China}
\email{liuruonan97v@163.com}
\author[Heyang Li]{Heyang Li}
\address[Heyang Li]{College of Mathematics and Information Science,
Henan Normal University, Xinxiang 453007, China}
\email{liheyang0619@163.com}


\begin{abstract}In this paper, we construct exact solutions that character three-dimensional, nonlinear trapped lee waves propagation superimposed on longitudinal atmospheric currents in the $\beta$-plane approximation. The solutions obtained are presented in Lagrangian coordinates, and are Gerstner-like solutions. In the process, we also derive the dispersion relation and analyze the density, pressure and the vorticity qualitatively.
\end{abstract}

\date{}

\maketitle




\section{Introduction}
\large
When stably stratified air ascends a mountain barrier, it produces oscillations, which can often trigger disturbances that propagate away from the mountain as gravity waves. Gravity waves triggered by the flow over a mountain are referred to as mountain waves or lee waves. The glossary of meteorology recognises two divisions of mountain waves: vertically propagating waves and trapped lee waves and the later ones are what we considered in this paper. Trapped lee waves have horizontal wavelengths of 5-35 km. They are trapped in a layer with high static stability and moderate wind speeds, usually in the lower troposphere beneath a thermal inversion and/or if stronger winds blow in the middle and upper troposphere \cite{Vo}. As the wave energy is trapped within the stable layer, these waves can propagate for long distances downwind of the mountain crest.

The problems for mountain waves have been addressed on numerical studies in recent papers \cite{Hi,Ji,Kr,Sh} and qualitative analysis of mountain waves can be found in \cite{Du,Te}, where linear theories are mostly used. While in view of the importance of nonlinear effects in large-amplitude flows, where mountain waves are usually of this type, it is more suitable to use the nonlinear governing equations (the Euler equation, the equation of mass conservation, the equation of state for the air and the first law of thermodynamics) to investigate the mountain waves, which have been derived systematically by Constantin in \cite{Co1}. This paper takes into account a further physical factor into the system, namely Coriolis forces introduced by the Earth's rotation, which bring additional terms to the Euler equation. The aim of this paper is to pursue a theoretical investigation of the nonlinear governing equations with full Coriolis forces for the mountain waves.

Due to the complexity and intractability of the governing equations, some simpler approximate models have been proposed to mitigate the Coriolis terms in the Euler equation, among which are the so-called $\beta$-plane and $f$-plane approximate models. In the $f$-plane approximation, the constant Coriolis forces parameters are considered and the latitude change is ignored. While in the $\beta$-plane approximation, the linear change of Coriolis forces parameters with latitude is introduced to allow for the variation of the Coriolis forces from point to point. The $f$-plane and $\beta$-plane approximate models have been applied widely in studying geophysical fluid dynamics. The readers can refer to \cite{CDY,Co7,Fa1,HM,Io1,LFW,M1} for geophysical water waves in the $f$-plane approximation and \cite{M2} for geophysical water waves in the $\beta$-plane, and so on. As for the atmospheric flows in the presence of Coriolis forces, the readers can refer to \cite{Ho,O} about the vertically propagating planetary waves and about the anelastic equatorial waves in the $\beta$-plane approximation, as well as \cite{Xu,YB} about gravity waves in the middle atmosphere and about large-scale tropical atmospheric dynamics in the $f$-plane approximation. In this study, the $\beta$-plane approximation is employed.

To describe the nonlinear dynamic of the given complex flows in detail, it is remarkable to find an exact solution to the mountain wave problem. While, it is known that finding explicit exact solutions is extremely difficult for the reason of the high intractability of the governing equations. As a matter of fact, the only known two-dimensional explicit exact solution to the governing equations was presented by Gerstner in $1802$ for periodic travelling gravity water waves with a non-flat free surface \cite{Ge}. Significant extensions of this solution to include the effects of Coriolis forces were given by Pollard in \cite{Po} and by Constantin in \cite{Co4}, where they constructed the nonlinear, three-dimensional solutions to the governing equations for geophysical ocean flows. In the past few decades, varieties of Gerstner-like and Pollard-like solutions were derived and analyzed to model a number of different physical and geophysical homogeneous inviscid fluid flows. For Gerstner-like solutions in the $\beta$-plane approximation, one can refer to \cite{Co4,Co6,Co3} for the equatorial trapped waves, \cite{H1} for the flows with underlying current, \cite{Io} about the geophysical edge waves, \cite{CIY,H2,H3} about the geophysical waves with centripetal forces and \cite{H4} for geophysical waves with a gravitational-correction term. For Gerstner-like solutions in the $f$-plane approximation, one can refer to \cite{H5,H6} for equatorial waves interacting with undercurrent, \cite{Fa} for geophysical trapped waves, \cite{Ma} for the  equatorial edge waves, and \cite{Co22} for wind-drift arctic flows outside the Amundsen
Basin. Recent studies about the Pollard-like solutions are investigated in \cite{CM,Io2,Mc} and so on.

Despite these advances in exact solutions modelling water waves, the investigation of similar exact solutions in the more intricate setting of compressible atmospheric flows remains comparatively rudimentary. The reasons lie in the vast differences in these two specific topics \cite{CJ1,CJ2,CJ3}. Very recently, Constantin showed that Gerstner's flow can be used to construct explicit mountain waves solutions propagating upwards in \cite{Co1}. Subsequently, Henry derived Gerstner-like solutions to model the trapped lee waves in the equatorial $f$-plane and vertically propagating mountain waves at general latitudes \cite{H24}.

Inspired by the recent studies \cite{Co1,H24}, the aim of this paper is to show that there exist exact atmospheric gravity wave solutions to the $\beta$-plane approximation. Within the Lagrangian framework, we construct nonlinear Gerstner-like solutions, which are three-dimensional because of the appearance of the Coriolis forces. In addition, qualitative analysis of the three-dimensional flow pattern is also given about the density and pressure, which depend not only on height but also on latitude compared with the ones in \cite{Co1,H24}, and we get that density and pressure decrease with height. In regard to the vorticity, the monotonicity of the second component has been investigated in \cite{Co1}. As the first and third component of the vorticity are both nonzero away from the Equator, we pay attention to the norm of vorticity and show that the magnitude of vorticity increases with height, which agrees with the observation that mountain waves can generate strong vortices at high altitudes.

The remainder of this paper is organized as follows. In Section 2, we present the governing equations with complete Coriolis forces for the mountain waves in the $\beta$-plane approximation. In Section 3, we propose the exact solution to the governing equations in the Lagrangian framework, which model the trapped lee waves. In Section 4, the density, pressure and vorticity are analyzed qualitatively.

\section{The governing equations in the $\beta$-plane approximation}
\large
Concerning mountain waves moving along the surface of Earth, we choose a rotating framework with the origin at a point on the Earth's surface: $(x, y, z)$ are Cartesian coordinates with the zonal coordinate $x$ pointing east, the meridional coordinate $y$ pointing north and the vertical coordinate $z$ pointing up. The Coriolis parameters, depending on the variable latitude $\phi$, are taken to be
\begin{equation*}
f=2\Omega\sin\phi,\quad \hat{f}=2\Omega\cos\phi,
\end{equation*}
with $\Omega=7.29\times10^{-5}\ rad/s$ being the rotational speed of the Earth. Within the $\beta$-plane approximation, we consider that, at the fixed latitude $\phi$, $\hat{f}$ is constant and $f$ has a linear variation in the latitude with the form of $f + \beta y$ for
$\beta=\frac{\hat{f}}{R}=\frac{2\Omega \cos \phi}{R}$, and $R=6378$ km being the equatorial radius. Then the full $\beta$-plane approximation governing equations for the mountain waves comprise the Euler equations \cite{Co1,Ho}
\begin{equation}\label{2.1}
\begin{cases}
\frac{\partial u}{\partial  t}+u\frac{\partial u}{\partial  x}+v\frac{\partial u}{\partial  y}+w\frac{\partial u}{\partial  z}+\hat{f}w-(f+\beta y)v&=-\frac 1 \rho P_x,\\
\frac{\partial v}{\partial  t}+u\frac{\partial v}{\partial  x}+v\frac{\partial v}{\partial  y}+w\frac{\partial v}{\partial  z}+(f+\beta y)u&=-\frac 1 \rho P_y,\\
\frac{\partial w}{\partial  t}+u\frac{\partial w}{\partial  x}+v\frac{\partial w}{\partial  y}+w\frac{\partial w}{\partial  z}-\hat{f}u&=-\frac 1 \rho P_z-g,
\end{cases}
\end{equation}
together with the equation of mass conservation
\begin{equation}\label{2.2}
\frac{\partial \rho}{\partial  t}+u\frac{\partial \rho}{\partial x}+v\frac{\partial \rho}{\partial y}+w\frac{\partial \rho}{\partial z}+\rho\left(\frac{\partial u}{\partial x}+\frac{\partial v}{\partial y}+\frac{\partial w}{\partial z}\right)=0,
\end{equation}
the equation of state for an ideal gas,
\begin{equation}\label{2.3}
P=\rho \mathfrak{R} \mathcal{T},
\end{equation}
and the first law of the thermodynamics
\begin{equation}\label{2.4}
c_p\left(\frac{\partial \mathcal{T}}{\partial t}+u\frac{\partial \mathcal{T}}{\partial x}+v\frac{\partial \mathcal{T}}{\partial y}+w\frac{\partial \mathcal{T}}{\partial z}\right)-\frac{1}{\rho}\left(\frac{\partial P}{\partial t}+u\frac{\partial P}{\partial x}+v\frac{\partial P}{\partial y}+w\frac{\partial P}{\partial z}\right)=0.
\end{equation}
Here $t$ is the time, $\mathbf{U}=(u,v,w)$ is the air velocity, $\rho$ is the air density, $P$ is the atmospheric pressure, $g =9.81\ m/s^2$ is the gravitational acceleration at the Earth's surface, $\mathcal{T}$ is the (absolute) temperature, $\mathfrak{R}\approx 287\ m^2 s^{-2} K^{-1}$ is the gas constant for dry air, $c_p\approx$ 1000 $m^2 s^{-2} K^{-1}$ is the specific heat of dry air at an atmospheric pressure of 1000 $mb$.
\section{Exact solutions for the trapped lee waves}
\large
The purpose of this section is to derive the exact explicit solutions to \eqref{2.1}-\eqref{2.4}. Within the Lagrangian framework, the Eulerian labelling variables $(x,y,z)$ at time $t$ are expressed as functions of Lagrangian labelling variables $(q,s,r)$, which specify moving air particles. Suppose that the positions of a particle at time $t$ is given by
\begin{equation}\label{3.1}
\begin{cases}
x=q-c_0t-\frac 1 k e^{k[r-m(s)]}\sin[k(q-ct)],\\
y=s,\\
z=Z_0+r+\frac 1 k e^{k[r-m(s)]}\cos[k(q-ct)],
\end{cases}
\end{equation}
where $k > 0$ is the wavenumber, $c > 0$ is the wave speed, $c_0\in\mr$  represents a mean background wind and the constant $Z_0$ can be freely chosen as a fixed reference altitude. The expression of the function $m$ depending on $s$ is to be determined below such that \eqref{3.1} defines an exact solution of the governing equations \eqref{2.1}-\eqref{2.4}.

The labelling variable $q$ parametrizes the leeward direction and $r$ captures the height of the layer of laminar flow, with the amplitude of the oscillations of a particle increasing with height. The label domain is given by real values of $(q,s,r)\in (\mr^+,[-s_0,s_0],(r_1,r_0))$ with
\begin{equation} \label{3.22}
r-m(s)<r_0-m(s)<0.
\end{equation}
The constraint $r_0-m(s)< 0$ has to be imposed for the reason of the boundedness of the vorticity of the air obtained in \eqref{4.3}. The flow pattern \eqref{3.1} models an atmospheric flow which is characteristic of trapped lee waves, for which air parcels propagate horizontally with velocity $(c_0,0,0)$ while oscillating vertically about a fixed mean altitude.
\begin{theorem}\label{the3.1}
Given
\begin{equation}\label{3.18}
m(s)=\frac {2fcs+\beta cs^2} {2(\hat{f}c_0+g)}, \quad \hat{f}c_0+g>0
\end{equation}
satisfying $r_0-m(s)< 0$ and a density function
\begin{equation}\label{3.20}
\rho(s,r)=F\left(\frac {e^{2k\left(r-\frac {2fcs+\beta cs^2} {2(\hat{f}c_0+g)}\right)}} {2k}-r+\frac {2fs+\beta s^2} {2(\hat{f}c_0+g)}c_0\right),
\end{equation}
where $F:(0,\infty)\rightarrow (0, \infty)$ is continuously differentiable and monotone increasing, then for a
given arbitrary wavenumber $k > 0$, there is a wave speed
\begin{equation}\label{3.17}
c=\frac{-\hat{f}+\sqrt{\hat{f}^2+4k(\hat{f}c_0+g)}}{2k}.
\end{equation}
and an associated pressure distribution
\begin{equation}\label{3.3}
P(s,r)=(\hat{f}c_0+g)\mathcal{F}\left(\frac{e^{2k\left(r-\frac{2fcs+\beta cs^2}{2(\hat{f}c_0+g)}\right)}}{2k}-r+\frac{2fs+\beta s^2}{2(\hat{f}c_0+g)}c_0\right),
\end{equation}
with $\mathcal{F}'=F$,  such that the velocity field $(u, v, w)$
\begin{equation}\label{3.7}
\begin{cases}
u=\frac {Dx} {Dt}=-c_0+c e^{\xi}\cos\theta,\\
v=\frac {Dy} {Dt}=0,\\
w=\frac {Dz} {Dt}=c e^{\xi}\sin\theta,
\end{cases}
\end{equation}
determined by \eqref{3.1} solves the system \eqref{2.1}-\eqref{2.4}.
\end{theorem}
\begin{proof}
To prove the theorem, we choose for notational convenience that
\begin{equation}\label{3.4}
\xi=k[r-m(s)],\quad \theta=k(q-ct),
\end{equation}
and accordingly compute the Jacobian matrix of \eqref{3.1}
\begin{equation}\label{3.5}
\begin{pmatrix}
\frac {\partial x}{\partial q}&\frac {\partial y}{\partial q}&\frac {\partial z}{\partial q}\\
\frac {\partial x}{\partial s}&\frac {\partial y}{\partial s}&\frac {\partial z}{\partial s}\\
\frac {\partial x}{\partial r}&\frac {\partial y}{\partial r}&\frac {\partial z}{\partial r}\\
\end{pmatrix}
=
\begin{pmatrix}
1-e^{\xi}\cos\theta&0&-e^{\xi}\sin\theta\\
m_se^{\xi}\sin\theta&1&-m_se^{\xi}\cos\theta\\
-e^{\xi}\sin\theta&0&1+e^{\xi}\cos\theta
\end{pmatrix}.
\end{equation}
The Jacobian determinant $D(s, r) = 1-e^{2\xi}$ of the map relating at the instant $t$ the particle positions to the labelling variables is time-independent, so \eqref{3.1} is area-preserving and thus
\begin{equation}\label{3.6}
\frac{\partial u}{\partial x}+\frac{\partial v}{\partial y}+\frac{\partial w}{\partial z}=0.
\end{equation}

The velocity \eqref{3.7} of a particle can also be calculated directly from \eqref{3.1} and acceleration
follows as
\begin{equation}\label{3.8}
\begin{cases}
\frac {Du} {Dt}=kc^2 e^{\xi}\sin\theta,\\
\frac {Dv} {Dt}=0,\\
\frac {Dw} {Dt}=-kc^2 e^{\xi}\cos\theta.
\end{cases}
\end{equation}
Considering the dependence of a density form on $(x,y,z,t)$-variables of the form
$\rho=\mathfrak{f}(X, Y, Z)$,
with
\begin{equation}\label{3.10}
\begin{cases}
X=x+c_0t-ct=\frac{\theta}{k}-\frac{1}{k}e^{\xi}\sin \theta,\\
Y=y=s,\\
Z=z-Z_0=r+\frac{1}{k}e^{\xi}\cos \theta,
\end{cases}
\end{equation}
then we have
\begin{align}\label{3.11}
\frac{\partial \rho}{\partial \theta}&=\frac{\partial \mathfrak{f}}{\partial X}\frac{\partial X}{\partial \theta}
+\frac{\partial \mathfrak{f}}{\partial Y}\frac{\partial Y}{\partial \theta}
+\frac{\partial \mathfrak{f}}{\partial Z}\frac{\partial Z}{\partial \theta}\nonumber\\
&=\frac{\partial \mathfrak{f}}{\partial X} \frac{1-e^{\xi}\cos \theta}{k}-\frac{\partial \mathfrak{f}}{\partial Z} \frac{e^{\xi}\sin \theta}{k}\nonumber\\
&=\frac{\partial \mathfrak{f}}{\partial X}\left(\frac{1}{k}-\frac{u+c_0}{ck}\right)-\frac{\partial \mathfrak{f}}{\partial Z}\left(\frac{w}{ck}\right)\nonumber\\
&=-\frac{1}{ck}\left[(u+c_0-c)\frac{\partial \mathfrak{f}}{\partial X}+w\frac{\partial \mathfrak{f}}{\partial Z}\right]\nonumber\\
&=-\frac{1}{ck}\left(\frac{\partial \rho}{\partial t}+u\frac{\partial \rho}{\partial x}+w\frac{\partial \rho}{\partial z}\right),
\end{align}
where the expressions of $(u,v,w)$ in \eqref{3.7} are used. Given that $\rho=\rho(s, r)$ in \eqref{3.20}, it is obvious from \eqref{3.11} that
\begin{equation}\label{3.12}
\frac{\partial \rho}{\partial t}+u\frac{\partial \rho}{\partial x}+w\frac{\partial \rho}{\partial z}=0.
\end{equation}
Combining \eqref{3.6}, \eqref{3.12} and \eqref{3.7}, the equation \eqref{2.2} holds.

Now we pay attention to the equations \eqref{2.1}, which, due to \eqref{3.7} and \eqref{3.8}, can be written as
\begin{equation}\label{3.13}
\begin{cases}
P_x=-\rho(kc^2+\hat{f}c)e^{\xi}\sin\theta,\\
P_y=-\rho\left((f+\beta s)ce^{\xi}\cos\theta-(f+\beta s)c_0\right),\\
P_z=-\rho\left(-(kc^2+\hat{f}c) e^{\xi}\cos\theta+\hat{f}c_0+g\right).
\end{cases}
\end{equation}
The change of variables
\begin{equation}\label{3.14}
\begin{pmatrix}
P_q\\P_s\\P_r
\end{pmatrix}
=\begin{pmatrix}
\frac {\partial x}{\partial q}&\frac {\partial y}{\partial q}&\frac {\partial z}{\partial q}\\
\frac {\partial x}{\partial s}&\frac {\partial y}{\partial s}&\frac {\partial z}{\partial s}\\
\frac {\partial x}{\partial r}&\frac {\partial y}{\partial r}&\frac {\partial z}{\partial r}\\
\end{pmatrix}
\begin{pmatrix}
P_x\\P_y\\P_z
\end{pmatrix},\\
\end{equation}
transforms \eqref{3.13} into
\begin{equation}\label{3.15}
\begin{cases}
P_q=-\rho(kc^2 +\hat{f}c-\hat{f}c_0-g)e^{\xi}\sin\theta,\\
P_s=-\rho(kc^2+\hat{f}c)m_se^{2\xi}-\rho\left((f+\beta s)c-(\hat{f}c_0+g)m_s\right)e^{\xi}\cos\theta\\
\quad\quad\:+\rho(f+\beta s)c_0,\\
P_r=\rho(kc^2+\hat{f}c)e^{2\xi}+\rho\left(kc^2 +\hat{f}c-\hat{f}c_0-g\right)e^{\xi}\cos\theta-\rho(\hat{f}c_0+g).
\end{cases}
\end{equation}
We prescribe now a suitable pressure function such that \eqref{3.15} hold, proving thus that \eqref{3.7} is indeed an exact solution of the Euler equation \eqref{2.1}. In the case that
\begin{equation}\label{3.16}
kc^2 +\hat{f}c-\hat{f}c_0-g=0,
\end{equation}
and
\[
m(s)=\frac {2fcs+\beta cs^2} {2(\hat{f}c_0+g)},
\]
where $c_0>-\frac g {\hat{f}}$ is valid on physical grounds (otherwise $c_0\leq -\frac g {2\Omega}\approx-6.7\times 10^{4} m/s$), then the equations \eqref{3.15} are reduced to
\begin{equation}\label{3.21}
\begin{cases}
P_q=0,\\
P_s=-\rho(kc^2+\hat{f}c)m_se^{2\xi}+\rho(f+\beta s)c_0,\\
P_r=\rho(kc^2+\hat{f}c)e^{2\xi}-\rho(\hat{f}c_0+g),
\end{cases}
\end{equation}
which satisfy the requirements $P_{rs}=P_{sr}$, $P_{qs}=P_{sq}$ and $P_{qr}=P_{rq}$. From \eqref{3.16} we get the
for $c>0$ the dispersion relation \eqref{3.17}. Moreover, given $\rho$ as in the form of \eqref{3.20}, we obtain the expression of the pressure as \eqref{3.3}.

On the other hand, with $\rho=\rho(s,r)$ and $P=P(s,r)$, the equation \eqref{2.3} ensures that $\mathcal{T}=\mathcal{T}(s,r)$, so that the temperature of an air parcel does not change during its motion. Therefore, \eqref{2.4} holds. This completes the proof.
\end{proof}
\begin{remark}
The solution obtained in Theorem \ref{the3.1} has no mean vertical velocity component, and thus it can not
prescribe upward propagating mountain waves in the $\beta$-plane approximation.
\end{remark}

\section{Qualitative Analysis}
\large
We compute the inverse of the Jacobian matrix \eqref{3.5} as
\begin{equation}\label{4.1}
\begin{pmatrix}
\frac {\partial q}{\partial x}&\frac {\partial s}{\partial x}&\frac {\partial r}{\partial x}\\
\frac {\partial q}{\partial y}&\frac {\partial s}{\partial y}&\frac {\partial r}{\partial y}\\
\frac {\partial q}{\partial z}&\frac {\partial s}{\partial z}&\frac {\partial r}{\partial z}
\end{pmatrix}
=\frac 1 {1-e^{2\xi}}
\begin{pmatrix}
1+e^{\xi}\cos\theta&0&e^{\xi}\sin\theta\\
-m_se^{\xi}\sin\theta&1-e^{2\xi}&m_s(e^{\xi}\cos\theta-e^{2\xi})\\
e^{\xi}\sin\theta&0&1-e^{\xi}\cos\theta
\end{pmatrix}.
\end{equation}
\subsection{Temperature, density and pressure}
Since $F$ referred to in Theorem \ref{the3.1} is monotone increasing, then we have by \eqref{3.22} that
\begin{equation*}
\frac{\partial \rho}{\partial r}=F^{'}\left(\frac {e^{2\xi}} {2k}-r+\frac {2fs+\beta s^2} {2(\hat{f}c_0+g)}c_0\right)
\left(e^{2\xi}-1\right)<0,
\end{equation*}
that is the density function $\rho=\rho(s,r)$ decreases with the $r$. Furthermore,  \eqref{4.1} and \eqref{3.1} reveal that
\begin{equation*}
\frac{\partial \rho}{\partial z}=\frac{\partial \rho}{\partial r}\frac{\partial r}{\partial z}
+\frac{\partial \rho}{\partial s}\frac{\partial s}{\partial z}
=\frac{\partial \rho}{\partial r}\cdot \frac{1-e^\xi \cos \theta}{1-e^{2\xi}} < 0,
\end{equation*}
which shows that the density decreases with height. Similarly, we have
\begin{equation*}
\frac{\partial P}{\partial z}=\frac{\partial P}{\partial r}\frac{\partial r}{\partial z}
+\frac{\partial P}{\partial s}\frac{\partial s}{\partial z}
=-\rho(\hat{f}c_0+g)(1-e^{\xi}\cos\theta)< 0,
\end{equation*}
so that the pressure decreases with altitude.
\subsection{The vorticity}
\large
According to \eqref{4.1}, we can get the velocity gradient tension
\begin{equation}\label{4.2}
\begin{aligned}
\nabla \mathbf{U}&=
\begin{pmatrix}
\frac {\partial u}{\partial x}&\frac {\partial u}{\partial y}&\frac {\partial u}{\partial z}\\
\frac {\partial v}{\partial x}&\frac {\partial v}{\partial y}&\frac {\partial v}{\partial z}\\
\frac {\partial w}{\partial x}&\frac {\partial w}{\partial y}&\frac {\partial w}{\partial z}
\end{pmatrix}
=\begin{pmatrix}
\frac {\partial u}{\partial q}&\frac {\partial u}{\partial s}&\frac {\partial u}{\partial r}\\
\frac {\partial v}{\partial q}&\frac {\partial v}{\partial s}&\frac {\partial v}{\partial r}\\
\frac {\partial w}{\partial q}&\frac {\partial w}{\partial s}&\frac {\partial w}{\partial r}
\end{pmatrix}
\begin{pmatrix}
\frac {\partial q}{\partial x}&\frac {\partial q}{\partial y}&\frac {\partial q}{\partial z}\\
\frac {\partial s}{\partial x}&\frac {\partial s}{\partial y}&\frac {\partial s}{\partial z}\\
\frac {\partial r}{\partial x}&\frac {\partial r}{\partial y}&\frac {\partial r}{\partial z}
\end{pmatrix}\\
&=\frac {cke^{\xi}} {1-e^{2\xi}}
\begin{pmatrix}
-\sin\theta&m_s(e^{\xi}-\cos\theta)&-e^{\xi}+\cos\theta\\
0&0&0\\
\cos\theta+e^{\xi}&-m_s\sin\theta&\sin\theta
\end{pmatrix}.
\end{aligned}
\end{equation}
The vorticity $\mathbf{\gamma}=(w_y-v_z, u_z-w_x,v_x-u_y):=(\gamma_1, \gamma_2, \gamma_3)$ is obtained as
\begin{align}\label{4.3}
(\gamma_1, \gamma_2, \gamma_3)=\left(-\frac {k(f+\beta s)c^2} {\hat{f}c_0+g} \frac {e^{\xi}\sin\theta} {1-e^{2\xi}},
-\frac {2kce^{2\xi}} {1-e^{2\xi}}, \frac {k(f+\beta s)c^2} {\hat{f}c_0+g} \frac{e^{\xi}\cos\theta-e^{2\xi}} {1-e^{2\xi}}\right).
\end{align}
Note that the $\gamma_2$, accounting for the local spin in the $y$-direction, is identical to the vorticity of a Gerstner wave. While the first and third component are both nonzero away from the Equator. The monotonicity of
$\gamma_2$ on the depth $z$ has been investigated in \cite{Co1}, and monotonicity of the magnitude of vorticity $\gamma$ is also of interest. As
\begin{equation}\label{4.4}
\frac{\partial |\gamma|}{\partial z}=\frac{\partial (\gamma_1^2+\gamma_2^2+\gamma_3^2)^{\frac{1}{2}}}{\partial z}
=\frac{1}{|\gamma|}\left(\gamma_1 \frac{\partial \gamma_1}{\partial z}+\gamma_2 \frac{\partial \gamma_2}{\partial z}+\gamma_3 \frac{\partial \gamma_3}{\partial z}\right),
\end{equation}
and a tedious calculation shows that
\begin{align}\label{4.5}
\gamma_1 \frac{\partial \gamma_1}{\partial z}+\gamma_2 \frac{\partial \gamma_2}{\partial z}+\gamma_3 \frac{\partial \gamma_3}{\partial z}=&
\frac{k^2c^4(f+\beta s)^2}{(\hat{f}c_0+g)^2} \cdot \frac{1-e^{\xi}\cos \theta}{(1-e^{2\xi})^4}
ke^{2\xi}\nonumber\\
&\times \left[1-3e^{\xi}\cos\theta+3e^{2\xi}-e^{3\xi}\cos\theta+\frac{8}{m_s^2}e^{2\xi}\right],
\end{align}
where \eqref{3.18}, \eqref{4.1} and \eqref{4.3} are used. Denoting
$A=3+\frac{8}{m_s^2},$ and defining the function
\begin{equation}\label{4.7}
\Psi(e^{\xi},\cos\theta)=1-3e^{\xi}\cos\theta+Ae^{2\xi}-e^{3\xi}\cos\theta,
\end{equation}
then for $\frac{k^2c^4(f+\beta s)^2}{(\hat{f}c_0+g)^2}>0$ and $\frac{1-e^{\xi}\cos \theta}{(1-e^{2\xi})^4}>0$, the key to the monotonicity of $|\gamma|$ on $z$ is to determine the sign of the function $\Psi(e^{\xi},\cos\theta)$. Note
\begin{equation}\label{4.8}
\Psi(e^{\xi},\cos\theta) \geq \Psi(e^{\xi})=1+e^\xi \left(Ae^\xi-e^{2\xi}-3\right).
\end{equation}
Moreover, from the arguments in \cite{Ho}, we know that magnitude of the Coriolis parameter $f$ and $\hat{f}$ is about $10^{-4}s^{-1}$, "beta" parameter $\beta$ is about $10^{-11}m^{-1}s^{-1}$, the speed of wave $c$ and mean background wind $c_0$ is typical of $3\times10m/s$, then if we take typically the latitude $s$ is of $10^4m$, we have $m_s=\frac {\left(f+\beta s\right)c} {\hat{f}c_0+g}\sim 3\times10^{-4}$ and thus $A=3+\frac{8}{m_s^2}\sim 8.89\times 10^{7}$. Furthermore, the graphs of $\Psi^{'}(e^{\xi})$ and $\Psi(e^{\xi})$ are sketched out in Figure 1.
\begin{figure}[htbp]
\centering
\subfigure[]
{\begin{minipage}[t]{0.48\textwidth}
\centering
\includegraphics[width=2.7in]{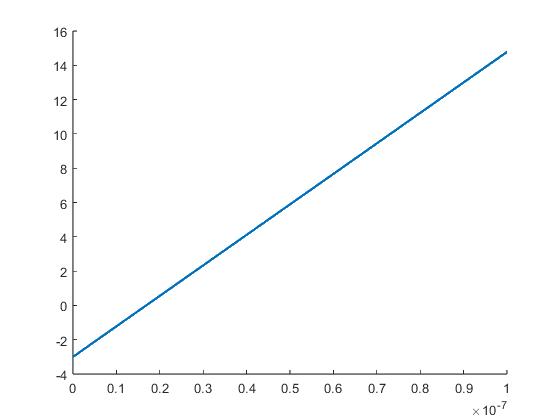}
\end{minipage}
}
\subfigure[]
{\begin{minipage}[t]{0.48\textwidth}
\centering
\includegraphics[width=2.7in]{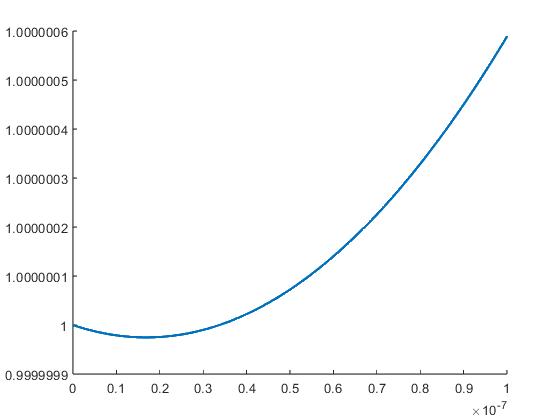}
\end{minipage}
}
\caption{(a) Typical graph of the map $\Psi'(X)=-3X^2+2AX+3$ with $X=e^\xi$ for $A=3+\frac{8}{m_s^2}$ and $m_s=3\times10^{-4}$. The zero of the map is $X_1=\frac{A-\sqrt{A^2-9}}{3}\sim 1.5\times 10^{-8}$.
(b) Typical graph of the map $\Psi (X)=1+X(AX-X^2-3)$ with $X=e^\xi$ for $A=3+\frac{8}{m_s^2}$ and $m_s=3\times10^{-4}$, where $\Psi (X_1)_{\min}=\frac{2A^3+(18-2A^2)\sqrt{A^2-9}-27A}{27}+1>0$.}
\end{figure}

Figure 1 shows that $\Psi(e^{\xi})>0$ and thus $\frac{\partial |\gamma|}{\partial z}>0$. That is to say, the magnitude of three-dimensional vorticity $\gamma$ increases with height. This indicates that lee waves can generate strong vortices at high altitudes, which agrees with the observation.

\vspace{0.5cm}
\noindent {\bf Acknowledgements.}
The work of Fan is supported by a NSF of Henan Province of China Grant No. 222300420478.

\end{document}